\documentclass[a4paper,11pt]{amsart}

\usepackage[utf8]{inputenc}
\usepackage[T1]{fontenc}
\usepackage{latexsym}
\usepackage{enumitem}
\usepackage{amsxtra}
\usepackage{amsmath}
\usepackage{amsfonts}
\usepackage{amssymb}
\usepackage{blkarray}
\usepackage{multirow}
\usepackage{mathrsfs}
\usepackage{mathdots}
\usepackage{pdfpages}
\usepackage{cite}
\usepackage{comment}
\usepackage{hyperref}
\usepackage[capitalise]{cleveref}

\newtheorem{theorem}{Theorem}[section]
\newtheorem{lemma}[theorem]{Lemma}

\newtheorem{proposition}[theorem]{Proposition}
\newtheorem{construction}[theorem]{Construction}

\newtheorem{sublemma}{}[theorem]
\newtheorem{conjecture}[theorem]{Conjecture}

\newenvironment{subproof}[1][\proofname]{%
  \begin{proof}[#1]%
}{%
  \end{proof}%
}
\numberwithin{equation}{section}

\theoremstyle{definition}

\theoremstyle{remark}

\newcommand{\seq}[1]{[#1]}
\newcommand{\ba}{\backslash}
\newcommand{\ttp}{$(s,2s,t,2t)$-property }

\DeclareMathOperator{\cl}{cl}

\setenumerate{label=\rm(\roman*),midpenalty=2}

\begin{document}
\title[Generalized Spikes]{Generalized spikes with circuits and cocircuits of different cardinalities}

\author[N.~Brettell]{Nick Brettell}
\address{School of Mathematics and Statistics\\
  Victoria University of Wellington\\
New Zealand}
\email{nick.brettell@vuw.ac.nz}
\author[K.~Grace]{Kevin Grace}
\address{Department of Mathematics\\
  Vanderbilt University\\
Nashville, Tennessee}
\email{kevin.m.grace@vanderbilt.edu}

\date{\today}

\begin{abstract}
  We consider matroids with the property that every subset of the ground set of size $s$ is contained in a $2s$-element circuit and every subset of size $t$ is contained in a $2t$-element cocircuit. We say that such a matroid has the \emph{$(s,2s,t,2t)$-property}.  A matroid is an \emph{$(s,t)$-spike} if there is a partition of the ground set into pairs such that the union of any $s$ pairs is a circuit and the union of any $t$ pairs is a cocircuit. Our main result is that all sufficiently large matroids with the $(s,2s,t,2t)$-property are $(s,t)$-spikes, generalizing a 2019 result that proved the case where $s=t$. We also present some properties of $(s,t)$-spikes.
\end{abstract}

\maketitle

\section{Introduction}

For integers $s$, $u$, $t$, and $v$, with $u \ge s \ge 1$ and $v \ge t \ge 1$, a matroid~$M$ has the \emph{$(s,u,t,v)$-property} if every $s$-element subset of $E(M)$ is contained in a circuit of size~$u$, and every $t$-element subset of $E(M)$ is contained in a cocircuit of size~$v$.
Matroids with this property appear regularly in the matroid theory literature: for example, wheels and whirls have the $(1,3,1,3)$-property, and (tipless) spikes have the $(2,4,2,4)$-property.
Note that $M$ has the $(s,u,t,v)$-property if and only if $M^*$ has the $(t,v,s,u)$-property.
Brettell, Campbell, Chun, Grace, and Whittle~\cite{bccgw2019} studied such matroids, and showed that if $u<2s$ or $v<2t$, then there are only finitely many matroids with the $(s,u,t,v)$-property~\cite[Theorem 3.3]{bccgw2019}.
On the other hand, in the case that $s=t$ and $u=v=2t$, any sufficiently large matroid with the $(s,u,t,v)$-property is a member of a class of structured matroids referred to as \emph{$t$-spikes}.  In particular, when $t=2$, this is the class typically known simply as \emph{(tipless) spikes}.

Our focus in this paper is also on the case where $u=2s$ and $v=2t$, but we drop the requirement that $s=t$. For positive integers $s$ and $t$, an \emph{$(s,t)$-spike} is a matroid on at least $2\max\{s,t\}$ elements whose ground set has a partition $(S_1,S_2,\ldots,S_n)$ into pairs such that the union of every set of $s$ pairs is a circuit and the union of every set of $t$ pairs is a cocircuit. The following is our main result:

\begin{theorem}
\label{mainthm}
There exists a function $f : \mathbb{N}^2 \rightarrow \mathbb{N}$ such that, if $M$ is a matroid with the \ttp and $|E(M)| \ge f(s,t)$, then $M$ is an $(s,t)$-spike.
\end{theorem}

\noindent
This proves the conjecture of Brettell et al.~\cite[Conjecture~1.2]{bccgw2019}. 


Our approach is essentially the same as in \cite{bccgw2019}, but some care is required to generalize the argument.
We note also that \cref{modcut} corrects an erroneous lemma \cite[Lemma 6.6]{bccgw2019}.

This paper is one in a developing series on matroids with the $(s,u,t,v)$-property.
First, Miller~\cite{miller2014} studied matroids with the $(2,4,2,4)$-property, proving the specialization of \cref{mainthm} to the case where $s=t=2$.  As previously mentioned, Brettell et al.~\cite{bccgw2019} considered the more general case where $s=t$ and $u=v=2t$, for any $t \ge 1$.
Oxley, Pfeil, Semple, and Whittle considered the case where $s=2$, $u=4$, $t=1$, and $v \in \{3,4\}$, showing that a sufficiently large $v$-connected matroid with the $(2,4,1,v)$-property is isomorphic to $M(K_{v,n})$ for some $n$~\cite{pfeil}.
A ``cyclic'' analogue of the $(s,u,t,v)$-property has also been considered, where a cyclic ordering $\sigma$ is imposed on $E(M)$, and only sets that appear consecutively with respect to $\sigma$ and have size~$s$ (or size~$t$) need appear in a circuit of size $u$ (or a cocircuit of size $v$, respectively).  The case where $s = u-1$ and $t = v-1$ and $s=t$ was considered by Brettell, Chun, Fife, and Semple~\cite{bcfs2019}; whereas Brettell, Semple, and Toft dropped the requirement that $s=t$~\cite{bst2022}.

This series of papers has been motivated by problems involving matroid connectivity. The well-known Wheels-and-Whirls Theorem of Tutte~\cite{tutte1966} states that wheels and whirls (which have the $(1,3,1,3)$-property) are the only $3$-connected matroids with no elements that can be either deleted or contracted to retain a $3$-connected matroid. Similarly, spikes (which have the $(2,4,2,4)$-property) are the only $3$-connected matroids on at least $13$ elements that have no triangles, no triads, and no pairs of elements that can be either deleted or contracted to preserve $3$-connectivity~\cite{williams2015}.

The following conjecture was stated as \cite[Conjecture 1.3]{bccgw2019}.  The case where $t=2$ was proved by Williams~\cite{williams2015}.
\begin{conjecture}
\label{conj:old}
  There exists a function $f : \mathbb{N} \rightarrow \mathbb{N}$ such that if $M$ is a $(2t-1)$-connected matroid with no circuits or cocircuits of size $2t-1$, and $|E(M)| \ge f(t)$, then either
  \begin{enumerate}
    \item there exists a $t$-element set $X \subseteq E(M)$ such that either $M/X$ or $M \ba X$ is $(t+1)$-connected, or
    \item $M$ is a $(t,t)$-spike.
  \end{enumerate}
\end{conjecture}

Indeed, sufficiently large $(t,t)$-spikes are $(2t-1)$-connected matroids~\cite[Lemma~6.5]{bccgw2019}, they have no circuits or cocircuits of size $(2t-1)$~\cite[Lemma~6.3]{bccgw2019}, and for every $t$-element subset $X \subseteq E(M)$, neither $M/X$ nor $M \ba X$ is $(t+1)$-connected. Optimistically, we offer the following generalization of \cref{conj:old}.

\begin{conjecture}
\label{conj:new}
  There exists a function $f : \mathbb{N}^2 \rightarrow \mathbb{N}$ such that if $M$ is a matroid with no circuits of size at most $2s-1$, no cocircuits of size at most $2t-1$, the matroid $M$ is $(2\min\{s,t\}-1)$-connected, and $|E(M)| \ge f(s,t)$, then either
  \begin{enumerate}
    \item there exists an $s$-element set $X \subseteq E(M)$ such that $M/X$ is $(s+1)$-connected,
    \item there exists a $t$-element set $X \subseteq E(M)$ such that $M \ba X$ is $(t+1)$-connected, or
    \item $M$ is an $(s,t)$-spike.
  \end{enumerate}
\end{conjecture}

\cref{sec:Preliminaries} recalls some terminology and a Ramsey-theoretic result used later in the paper. In \cref{sec:echidnas}, we recall the definition of echidnas from~\cite{bccgw2019} and show that every matroid with the $(s,2s,t,2t)$-property and having a sufficiently large $s$-echidna is an $(s,t)$-spike. In \cref{sec:t2t}, we prove \cref{mainthm}. Finally, \cref{sec:tspikeprops} describes some properties of $(s,t)$-spikes, as well as a construction that allows us to build an $(s,t+1)$-spike from an $(s,t)$-spike.

\section{Preliminaries}
\label{sec:Preliminaries}

Our notation and terminology follows Oxley~\cite{oxbook}.
We refer to the fact that a circuit and a cocircuit cannot intersect in exactly one element as ``orthogonality''.
A set $S_1$ \emph{meets} a set $S_2$ if $S_1 \cap S_2 \neq \emptyset$.
We denote $\{1,2,\dotsc,n\}$ by $\seq{n}$, and, for positive integers $i < j$, we denote $\{i,i+1,\dotsc,j\}$ by $[i,j]$.
We denote the set of positive integers by $\mathbb{N}$.

In order to prove \cref{mainthm}, we will use some hypergraph Ramsey Theory~\cite{ramsey1930}. Recall that a hypergraph is \emph{$k$-uniform} if every hyperedge has size~$k$.

\begin{theorem}[Ramsey's Theorem for $k$-uniform hypergraphs]
  \label{hyperramsey}
  For positive integers $k$ and $n$, there exists an integer $r_k(n)$ such that if $H$ is a $k$-uniform hypergraph on $r_k(n)$ vertices, then $H$ has either a clique on $n$ vertices, or a stable set on $n$ vertices.
\end{theorem}

\section{Echidnas and \texorpdfstring{$(s,t)$}{(s,t)}-spikes}
\label{sec:echidnas}

Recall that $M$ is an $(s,t)$-spike if there is a partition of $E(M)$ into pairs such that the union of any $s$ pairs is a circuit and the union of any $t$ pairs is a cocircuit.
In this section, we prove a sufficient condition for $M$ to be an $(s,t)$-spike.  Namely, we prove as \cref{lem:swamping} that if $M$ has the $(s,2s,t,2t)$-property, and a subset of $E(M)$ can be partitioned into $u$ pairs such that the union of any $t$ pairs is a circuit, then, when $u$ is sufficiently large, $M$ is an $(s,t)$-spike.  Conforming with \cite{bccgw2019}, we call such a partition a $t$-echidna, as defined below.

  Let $M$ be a matroid.
  A $t$-\emph{echidna} of order $n$ is a partition $(S_1,\ldots, S_n)$ of a subset of $E(M)$ such that 
  \begin{enumerate}
    \item $|S_i|=2$ for all $i \in \seq{n}$, and 
    \item $\bigcup_{i \in I}S_i$ is a circuit for all $I \subseteq \seq{n}$ with $|I|=t$.
  \end{enumerate}
  For $i \in \seq{n}$, we say $S_i$ is a \emph{spine}.
  We say $(S_1,\ldots,S_n)$ is a \emph{$t$-coechidna} of $M$ if $(S_1,\ldots,S_n)$ is a $t$-echidna of $M^*$.

Let $(S_1,\dotsc,S_n)$ be a $t$-echidna of a matroid $M$.
If $(S_1,\dotsc,S_m)$ is a $t$-echidna of $M$, for some $m \ge n$, we say that $(S_1,\dotsc,S_n)$ \emph{extends} to $(S_1,\dotsc,S_m)$.
We say that $\pi=(S_1,\dotsc,S_n)$ is \emph{maximal} if $\pi$ extends only to $\pi$.

Note that a matroid~$M$ is an $(s,t)$-spike if there exists a partition $\pi=(A_1,\ldots,A_m)$ of $E(M)$ such that $\pi$ is an $s$-echidna and a $t$-coechidna, for some $m\geq\max\{s,t\}$. In this case, we say that the $(s,t)$-spike~$M$ has \emph{order~$m$}, we call $\pi$ the \emph{associated partition} of the $(s,t)$-spike~$M$, and we say that $A_i$ is an \emph{arm} of the $(s,t)$-spike for each $i \in \seq{m}$. An $(s,t)$-spike with $s=t$ is also called a \emph{$t$-spike}.
Note that if $M$ is an $(s,t)$-spike, then $M^*$ is a $(t,s)$-spike.

Throughout this section, we assume that $s$ and $t$ are positive integers.

\begin{lemma}
  \label{lem:coechidna}
  Let $M$ be a matroid with the $(s,2s,t,2t)$-property.
  If $M$ has an $s$-echidna $(S_1,\ldots, S_n)$, where $n\geq s+2t-1$, then $(S_1,\ldots, S_n)$ is also a $t$-coechidna of $M$.
\end{lemma}

\begin{proof}
  Suppose $M$ has an $s$-echidna $(S_1,\ldots, S_n)$ with $n \ge s+2t-1$, and
  let $S_i=\{x_i,y_i\}$ for each $i \in [n]$. We show, for every $t$-element subset $J$ of $[n]$, that $\bigcup_{j \in J} S_j$ is a cocircuit. Without loss of generality, let $J=[t]$. By the $(s,2s,t,2t)$-property, $\{x_1,\ldots,x_{t}\}$ is contained in a $2t$-element cocircuit~$C^*$. Suppose for a contradiction that $C^*\neq\bigcup_{j \in J} S_j$.
  Then there is some $i \in [t]$ such that $y_i\notin C^*$. Without loss of generality, say $y_1\notin C^*$.
  
  Let $I$ be an $(s-1)$-element subset of $[t+1,n]$.
  For any such $I$, the set $S_1 \cup \bigcup_{i \in I} S_i$ is a circuit that meets $C^*$. By orthogonality, $\bigcup_{i \in I} S_i$ meets $C^*$. 
  Thus, $C^*$ avoids at most $s-2$ of the $S_i$'s for $i \in [t+1,n]$. In fact, as $C^*$ meets each $S_i$ with $i \in [t]$, the cocircuit~$C^*$ avoids at most $s-2$ of the $S_i$'s for $i \in [n]$. Thus $|C^*| \ge n-(s-2) \ge (s+2t-1) -(s-2) =2t+1 > 2t$, a contradiction.
  
  Therefore, we conclude that $C^*=\bigcup_{j \in J} S_j$, and the result follows.
\end{proof}



\sloppy
\begin{lemma}
  \label{lem:rep-orthog}
  Let $M$ be a matroid with the $(s,2s,t,2t)$-property, and let $(S_1,\ldots, S_n)$ be an $s$-echidna of $M$ with $n\geq\max\{s+2t,2s+t\}-1$.
  
  \begin{itemize}
  \item[(i)]
  Let $I$ be an $(s-1)$-subset of $[n]$. For $z\in E(M)-\bigcup_{i \in I}S_i$, there is a $2s$-element circuit containing $\{z\} \cup \bigcup_{i \in I}S_i$.
  \item[(ii)]
  Let $I$ be a $(t-1)$-subset of $[n]$. For $z\in E(M)-\bigcup_{i \in I}S_i$, there is a $2t$-element cocircuit containing $\{z\} \cup \bigcup_{i \in I}S_i$.
\end{itemize}
\end{lemma}
\fussy

\begin{proof}
  First we prove (i). For $i \in [n]$, let $S_i=\{x_i,y_i\}$.
  By the $(s,2s,t,2t)$-property, there is a $2s$-element circuit~$C$ containing $\{z\} \cup \{x_i : i \in I\}$.
  Let $J$ be a $(t-1)$-element subset of $[n]$ such that $C$ and $\bigcup_{j \in J}S_j$ are disjoint (such a set exists since $|C|=2s$ and $n \ge 2s+t-1$).
  For $i \in I$, let $C^*_i=S_i \cup \bigcup_{j \in J} S_j$, and observe that $x_i \in C^*_i \cap C$, and $C^*_i \cap C \subseteq S_i$.
  By \cref{lem:coechidna}, $(S_1,\dotsc,S_n)$ is a $t$-coechidna as well as an $s$-echidna; therefore, $C^*_i$ is a cocircuit.
  Now, for each $i \in I$, orthogonality implies that $|C^*_i \cap C| \ge 2$, and hence $y_i \in C$.
  So $C$ contains $\{z\} \cup \bigcup_{i \in I}S_i$, as required.
  
  Now, to prove (ii), recall that $(S_1,\dotsc,S_n)$ is a $t$-coechidna by Lemma \cref{lem:coechidna}. Therefore, (ii) follows by (i) and duality.
\end{proof}

\begin{lemma}
  \label{lem:swamping}
  Let $M$ be a matroid with the $(s,2s,t,2t)$-property.
  If $M$ has an $s$-echidna $\pi=(S_1,\ldots, S_n)$, where $n\geq\max\{s+2t-1,2s+t-1,3s+t-3\}$, then $(S_1,\ldots, S_n)$ extends to a partition of $E(M)$ that is both an $s$-echidna and a $t$-coechidna. 
\end{lemma}

\begin{proof}
  Let $\pi'=(S_1, \dotsc, S_m)$ be a maximal $s$-echidna with $X=\bigcup_{i = 1}^{m} S_i\subseteq E(M)$.
  Suppose for a contradiction that $X\neq E(M)$.
  Since $\pi'$ is maximal, $m\geq n\geq s+2t-1$.
  Therefore, by Lemma \ref{lem:coechidna}, $\pi'$ is a $t$-coechidna.
  
  Let $z\in E(M)-X$.
  By Lemma \ref{lem:rep-orthog}, there is a $2s$-element circuit $C = (\bigcup_{i \in [s-1]} S_i)\cup \{z,z'\}$ for some $z'\in E(M)$.
  We claim that $z'\notin X$.
  Towards a contradiction, suppose that $z'\in S_k$ for some $k\in [s,m]$.
  Let $J$ be a $t$-element subset of $[s,m]$ containing $k$.
  Then, since $(S_1,\dotsc,S_m)$ is a $t$-coechidna, $\bigcup_{j \in J}S_j$ is a cocircuit that contains $z'$.
  Now, this cocircuit intersects the circuit~$C$ in a single element $z'$, contradicting orthogonality.
  Thus, $z'\notin X$, as claimed.

  We next show that $(\{z,z'\}, S_{s}, S_{s+1}, \ldots, S_m)$ is a $t$-coechidna.
  Since $\pi'$ is a $t$-coechidna, it suffices to show that $\{z,z'\} \cup \bigcup_{i \in I}S_i$ is a cocircuit for each $(t-1)$-element subset~$I$ of $[s,m]$.
  Let $I$ be such a set.
  \Cref{lem:rep-orthog} implies that there is a $2t$-element cocircuit~$C^*$ of $M$ containing $\{z\} \cup \bigcup_{i\in I}S_i$.
  By orthogonality, $|C\cap C^*|>1$. Therefore, $z'\in C^*$. Thus, $(\{z,z'\}, S_{s}, S_{s+1}, \ldots, S_m)$ is a $t$-coechidna.
  Since this $t$-coechidna has order $1+m-(s-1)\geq n-s+2\geq2s+t-1$, the dual of \cref{lem:coechidna} implies that $(\{z,z'\}, S_{s}, S_{s+1}, \dotsc, S_m)$ is also an $s$-echidna.
  
  Next we show that $(\{z,z'\}, S_1, S_2, \dotsc, S_m)$ is a $t$-coechidna. Let $I$ be a $(t-1)$-element subset of $[m]$. We claim that $\{z,z'\} \cup \bigcup_{i \in I}S_i$ is a cocircuit.
  Let $J$ be an $(s-1)$-element subset of $[s,m]-I$.
  Then $C=\{z,z'\} \cup \bigcup_{j \in J}S_j$ is a circuit since $(\{z,z'\}, S_{s}, S_{s+1}, \dotsc, S_m)$ is an $s$-echidna.
  By \cref{lem:rep-orthog}, there is a $2t$-element cocircuit~$C^*$ containing $\{z\} \cup \bigcup_{i \in I}S_i$.
  By orthogonality between $C$ and $C^*$, we have $z'\in C^*$.
  Since $I$ was arbitrarily chosen, $(\{z,z'\}, S_1, S_2, \dotsc, S_m)$ is a $t$-coechidna.
  By the dual of \cref{lem:coechidna}, it is also an $s$-echidna, contradicting the maximality of $(S_1,\dotsc,S_m)$.
\end{proof}

\section{Matroids with the \texorpdfstring{$(s,2s,t,2t)$}{(s,2s,t,2t)}-property}
\label{sec:t2t}

In this section, we prove that every sufficiently large matroid with the \ttp is an $(s,t)$-spike. We will show that a sufficiently large matroid with the \ttp has a large $s$-echidna or $t$-coechidna;
it then follows, by \cref{lem:swamping}, that the matroid is an $(s,t)$-spike. As in the previous section, we assume that $s$ and $t$ are positive integers.

\begin{lemma}
  \label{lem:rank-t}
  Let $M$ be a matroid with the $(s,2s,t,2t)$-property, and let $X\subseteq E(M)$.
  \begin{enumerate}
    \item If $r(X)<s$, then $X$ is independent.\label{rt1}
    \item If $r(X)=s$, then $M|X\cong U_{s,|X|}$ and $|X|<s+2t$.\label{rt2}
  \end{enumerate}
\end{lemma}

\begin{proof}
  Every subset of $E(M)$ of size at most $s$ is independent since it is contained in a circuit of size $2s$. In particular, \ref{rt1} holds.
  
  Now let $r(X)=s$.  Then every $(s+1)$-element subset of $X$ is a circuit, so $M|X\cong U_{s,|X|}$.
  Suppose for a contradiction that $|X|\geq s+2t$. Let $C^*$ be a $2t$-element cocircuit such that there is some $x\in X\cap C^*$. Then $X-C^*$ is contained in the hyperplane $E(M)-C^*$. Since $x\in X\cap C^*$, we have $r(X-C^*)<r(X)=s$. Therefore, $X-C^*$ is an independent set, so $|X-C^*|<s$. Since $|X|\geq s+2t$, we have $|C^*|>2t$, a contradiction. Thus, \ref{rt2} holds.
\end{proof}

\begin{lemma}
  \label{lemmaA}
  Let $M$ be a matroid with the $(s,2s,t,2t)$-property, and let $C_1^*,C_2^*,\dotsc,C_{s-1}^*$ be a collection of pairwise disjoint cocircuits of $M$. 
  Let $Y = E(M)-\bigcup_{i \in [s-1]} C_i^*$.
  For all $y \in Y$, there is a $2s$-element circuit~$C_y$ containing $y$ such that either
  \begin{enumerate}
    \item $|C_y \cap C_i^*| = 2$ for all $i \in [s-1]$, or\label{A1}
    \item $|C_y \cap C_j^*| = 3$ for some $j \in [s-1]$, and $|C_y \cap C_i^*| = 2$ for all $i \in [s-1]-\{j\}$.\label{A2}
  \end{enumerate}
  Moreover, if $C_y$ satisfies \ref{A2}, then there are at most $s+2t-1$ elements $w \in Y$ such that $(C_y-y) \cup \{w\}$ is a circuit.
\end{lemma}

\begin{proof}
  Choose an element $c_i \in C_i^*$ for each $i \in [s-1]$.  By the $(s,2s,t,2t)$-property, there is a $2s$-element circuit~$C_y$ containing $\{c_1,c_2,\dotsc,c_{s-1},y\}$, for each $y \in Y$.
  By orthogonality, $C_y$ satisfies \ref{A1} or \ref{A2}.

  Suppose $C_y$ satisfies \ref{A2}, and let $S =C_y-Y= C_y-\{y\}$.
  Let $W = \{w \in Y : S \cup \{w\} \textrm{ is a circuit}\}$.
  It remains to prove that $|W| < s+2t$.
  Observe that $W \subseteq \cl(S) \cap Y$, and, since $S$ contains $s-1$ elements in pairwise disjoint cocircuits that avoid $Y$, we have $r(\cl(S) \cup Y) \ge r(Y) + (s-1)$.  Thus,
  \begin{align*}
    r(W) &\le r(\cl(S) \cap Y) \\
    &\le r(\cl(S)) + r(Y) - r(\cl(S) \cup Y) \\
    &\le (2s-1) + r(Y) - (r(Y)+ (s-1)) \\
    &=s,
  \end{align*}
  using submodularity of the rank function at the second line.

  Now, by \cref{lem:rank-t}\ref{rt1}, if $r(W) < s$, then $W$ is independent, so $|W| = r(W) < s < s + 2t$.
  On the other hand, by \cref{lem:rank-t}\ref{rt2}, if $r(W)=s$, then $M|W \cong U_{t,|W|}$ and $|W|<s+2t$, as required.
\end{proof}

\begin{lemma}
  \label{lem:disjoint}
  There exists a function $h$ such that if $M$ is a matroid with at least $h(k,d,t)$ $k$-element circuits, and the property that every $t$-element set is contained in a $2t$-element cocircuit for some positive integer $t$, then $M$ has a collection of $d$ pairwise disjoint $2t$-element cocircuits.
\end{lemma}

\begin{proof}
  By \cite[Lemma 3.2]{bccgw2019}, there is a function $g$ such that if $M$ has at least $g(k,d)$ $k$-element circuits, then $M$ has a collection of $d$ pairwise disjoint circuits.
  We define $h(k,d,t) = g(k,dt)$, and claim that a matroid with at least $h(k,d,t)$ $k$-element circuits, and the property that every $t$-element set is contained in a $2t$-element cocircuit, has a collection of $d$ pairwise disjoint $2t$-element cocircuits.

  Let $M$ be such a matroid.
  Then $M$ has a collection of $dt$ pairwise disjoint circuits.
  We partition these into $d$ groups of size $t$: 
  call this partition $(\mathcal{C}_1,\dotsc,\mathcal{C}_d)$.
  Since the $t$ circuits in any cell of this partition are pairwise disjoint, it now suffices to show that, for each $i \in [d]$, there is a $2t$-element cocircuit contained in the union of the members of $\mathcal{C}_i$.
  Let $\mathcal{C}_i = \{C_1,\dotsc,C_{t}\}$ for some $i \in [d]$.
  Pick some $c_j \in C_j$ for each $j \in [t]$.
  Then, since $\{c_1,c_2,\dotsc,c_{t}\}$ is a $t$-element set, it is contained in a $2t$-element cocircuit, which, by orthogonality, is contained in $\bigcup_{j \in [t]}C_j$.
\end{proof}

\begin{lemma}
  \label{setup}
  Let $M$ be a matroid with the $(s,2s,t,2t)$-property such that $r(M)\geq r^*(M)$. There exists a function $g$ such that, if $|E(M)| \ge g(s,t,q)$, then $M$ has $s-1$ pairwise disjoint $2t$-element cocircuits $C_1^*, C_2^*, \dotsc, C_{s-1}^*$, and
  there is some $Z \subseteq E(M)-\bigcup_{i \in [s-1]}C_i^*$ such that
  \begin{enumerate}
    \item $r_{M}(Z) \ge q$, and\label{ps1}
    \item for each $z \in Z$, there exists an element $z'\in Z-\{z\}$ such that $\{z,z'\}$ is contained in a $2s$-element circuit~$C$ with $|C \cap C_i^*|=2$ for each $i \in [s-1]$.\label{ps2}
  \end{enumerate}
\end{lemma}

\begin{proof}
  By \cref{lem:disjoint}, there is a function $h$ such that if $M$ has at least $h(k,d,t)$ $k$-element circuits, then $M$ has $d$ pairwise disjoint $2t$-element cocircuits. 
  
  Suppose $|E(M)|\geq 2s\cdot h(2s,s-1,t)$. By the $(s,2s,t,2t)$-property, $M$ has at least $h(2s,s-1,t)$ distinct $2s$-element circuits. Therefore, by \cref{lem:disjoint}, $M$ has a collection of $s-1$ pairwise disjoint $2t$-element cocircuits $C_1^*,\dotsc, C_{s-1}^*$.

  Let $X = \bigcup_{i \in [s-1]}C_i^*$ and $Y=E(M)-X$.
  By \cref{lemmaA}, for each $y \in Y$ there is a $2s$-element circuit~$C_y$ containing $y$ such that $|C_y \cap C_j^*| = 3$ for at most one $j \in [s-1]$ and $|C_y \cap C_i^*| = 2$ otherwise.
  Let $W$ be the set of all $w \in Y$ such that $w$ is in a $2s$-element circuit~$C$ with $|C\cap C_j^*|=3$ for some $j \in [s-1]$, and $|C \cap C_i^*|=2$ for all $i \in [s-1]-\{j\}$.
  Now, letting $Z=Y-W$, we see that \ref{ps2} is satisfied.  It remains to show that \ref{ps1} holds.
  
  Since each $C_i^*$ has size $2t$, there are $(s-1)\binom{2t}{3}\binom{2t}{2}^{s-2}$
  sets $X'\subseteq X$ with $|X' \cap C_j^*|=3$ for some $j \in [s-1]$ and $|X' \cap C_i^*|=2$ for all $i \in [s-1]-\{j\}$.
  It follows, by \cref{lemmaA}, that $|W| \le f(s,t)$ where \[f(s,t) = (s+2t-1)\left[(s-1)\binom{2t}{3}\binom{2t}{2}^{s-2}\right].\]

  We define \[g(s,t,q) = \max\left\{2s\cdot h(2s,s-1,t), 2\big(2t(s-1)+f(s,t)+q\big)\right\}.\]
  Suppose that $|E(M)| \ge g(s,t,q)$.
  Since $r(M)\geq r^*(M)$ and $|E(M)|\geq2(2t(s-1)+f(s,t)+q)$, we have $r(M) \ge 2t(s-1)+f(s,t)+q$. Then,
  \begin{align*}
    r_{M}(Z) &\ge r_{M}(Y) - |W| \\
    &\ge \big(r(M)-2t(s-1)\big) - f(s,t) \\
    &\ge q,
  \end{align*}
  so \ref{ps1} holds as well.
\end{proof}

\sloppy
\begin{lemma}
  \label{lem:payoff}
  Let $M$ be a matroid with the $(s,2s,t,2t)$-property. Suppose $M$ has $s-1$ pairwise disjoint $2t$-element cocircuits $C_1^*, C_2^*, \dotsc, C_{s-1}^*$
  and, for some positive integer~$p$, there is a set $Z \subseteq E(M)-\bigcup_{i \in [s-1]}C_i^*$ such that
  \begin{enumerate}[label=\rm(\alph*)]
    \item $r(Z) \ge \binom{2t}{2}^{s-1}(p + 2(s-1))$, and
    \item for each $z \in Z$, there exists an element $z'\in Z-\{z\}$ such that $\{z,z'\}$ is contained in a $2s$-element circuit $C$ of $M$ with $|C\cap C_i^*|=2$ for each $i\in [s-1]$.
  \end{enumerate}
  There exists a subset $Z' \subseteq Z$ and a partition $\pi=( Z_1', \dotsc, Z_p' )$ of $Z'$ into pairs such that
  \begin{enumerate}
    \item each circuit of $M|Z'$ is a union of pairs in $\pi$, and
    \item the union of any $s$ pairs in $\pi$ contains a circuit.
  \end{enumerate}
\end{lemma}
\fussy

\begin{proof}
  We first prove the following:

  \begin{sublemma}
    \label{prelem:payoff}
    There exists a $(2s-2)$-element set $X$ such that $|X\cap C_i^*|=2$ for every $i\in[s-1]$ and a set $Z'\subseteq Z$ with a partition $\pi=\{ Z_1', \dotsc, Z_p' \}$ of $Z'$ into pairs such that \begin{enumerate}[label=\rm(\Roman*)]
      \item $X \cup Z_i'$ is a circuit, for each $i \in [p]$ and \label{ppo1}
      \item $\pi$ partitions the ground set of $(M/X)|Z'$ into parallel classes such that $r_{M/X}\big(\bigcup_{i \in [p]}Z_i'\big)=p$. \label{ppo2}
    \end{enumerate}
  \end{sublemma}
    
  \begin{subproof}
    By (b), for each $z \in Z$, there exists an element $z'\in Z-\{z\}$ and a set $X'$ such that $\{z,z'\} \cup X'$ is a circuit of $M$ and $X'$ is the union of pairs $Y_i$ for $i\in[s-1]$, with $Y_i\subseteq C_i^*$.
    Since $|C_i^*|=2t$ for each $i\in[s-1]$, there are $\binom{2t}{2}^{s-1}$ choices for $(Y_1,Y_2,\ldots,Y_{s-1})$. Therefore, for some $m\leq\binom{2t}{2}^{s-1}$, there are $(2s-2)$-element sets $X_1,X_2,\ldots,X_m$, and sets $Z_1,Z_2,\ldots,Z_m$ whose union is $Z$, such that each of $X_1,X_2,\ldots,X_m$ intersects $C_i^*$ in two elements for each $i\in[s-1]$, and such that, for each $j\in[m]$ and each $z_j\in Z_j$, there is an element $z_j'$ such that $\{z_j,z_j'\}\cup X_j$ is a circuit. Since $Z=\bigcup_{i \in [m]}Z_i$, we have $\sum_{i\in[m]}r(Z_i)\geq r(Z)$. Thus, the pigeonhole principle implies that there is some $j\in[m]$ such that \[r(Z_j) \ge \frac{r(Z)}{\binom{2t}{2}^{s-1}} \ge p+2(s-1),\] by (a).
    
    We define $Z' = Z_j$ and $X = X_j$.
    Observe that $X \cup \{z,z'\}$ is a circuit, for some pair $\{z,z'\} \subseteq Z'$, if and only if $\{z,z'\}$ is a parallel pair in $M/X$.
    Therefore, there is a partition of the ground set of $(M/X)|Z'$ into parallel classes, where every parallel class has size at least two.
    Let $\{\{z_1,z_1'\}, \dotsc,\{z_n,z_n'\}\}$ be a collection of pairs from each parallel class such that $\{z_1,z_2,\dotsc,z_n\}$ is an independent set in $(M/X)|Z'$.
    Note that $n\geq r_{M/X}(Z') = r(Z' \cup X) -r(X) \ge r(Z') - 2(s-1) \ge p$. For $i\in[p]$, let $Z_i'=\{z_i,z_i'\}$. Then $\pi=\{ Z_1', \dotsc, Z_p' \}$ satisfies \ref{prelem:payoff}.
  \end{subproof}

  Let $X$, $\pi$, and $Z'$ be as described in \ref{prelem:payoff}, and let $\mathcal{X} = \{X_1,\dotsc,X_{s-1}\}$, where $X_i = \{x_i,x_i'\} = X \cap C_i^*$.
  
  \begin{sublemma}
    \label{metamatroid}
    Each circuit of $M|(X \cup Z')$ is a union of pairs in $\mathcal{X} \cup \pi$.
  \end{sublemma}

  \begin{subproof}
    Let $C$ be a circuit of $M|(X \cup Z')$.
    If $x_i \in C$, for some $\{x_i,x_i'\} \in \mathcal{X}$, then orthogonality with $C_i^*$ implies that $x_i' \in C$.
    Assume for a contradiction that $\{z,z'\} \in \pi$ and $C \cap \{z,z'\} = \{z\}$.
    Let $W$ be the union of the pairs in $\pi$ containing elements of $(C-\{z\}) \cap Z'$.  Then $z \in \cl(X \cup W)$.  Hence $z \in \cl_{M/X}(W)$, contradicting \cref{prelem:payoff}\ref{ppo2}.
  \end{subproof}

  \begin{sublemma}
    \label{induct}
    Every union of $s$ pairs in $\mathcal{X} \cup \pi$ contains a circuit.
  \end{sublemma}
  
  \begin{subproof}
    Let $\mathcal{W}$ be a subset of $\mathcal{X} \cup \pi$ of size $s$.
    We proceed by induction on the number of pairs in $\mathcal{W} \cap \pi$.
    If there is only one pair in $\mathcal{W} \cap \pi$, then the union of the pairs in $\mathcal{W}$ contains a circuit (indeed, is a circuit) by \cref{prelem:payoff}\ref{ppo1}.
    Suppose the result holds for any subset containing $k$ pairs in $\pi$, and let $\mathcal{W}$ be a subset containing $k+1$ pairs in $\pi$.
    Let $\{x,x'\}$ be a pair in $\mathcal{X}-\mathcal{W}$,
    and let $W = \bigcup_{W' \in \mathcal{W}}W'$.
    Then $W \cup \{x,x'\}$ is the union of $s+1$ pairs of $\mathcal{X} \cup \pi$, of which $k+1$ are in $\pi$, so, by the induction hypothesis, $W \cup \{x,x'\}$ properly contains a circuit~$C_1$.
    If $\{x,x'\} \subseteq E(M)-C_1$, then $C_1 \subseteq W$, in which case the union of the pairs in $\mathcal{W}$ contains a circuit, as desired.
    Therefore, we may assume, by \cref{metamatroid}, that $\{x,x'\} \subseteq C_1$.
    Since $X$ is independent, there is a pair $\{z,z'\} \subseteq Z' \cap C_1$.
    By the induction hypothesis, there is a circuit~$C_2$ contained in $(W-\{z,z'\}) \cup \{x,x'\}$.
    Observe that $C_1$ and $C_2$ are distinct, and $\{x,x'\} \subseteq C_1 \cap C_2$.
    Circuit elimination on $C_1$ and $C_2$, and \cref{metamatroid}, imply that there is a circuit $C_3 \subseteq (C_1 \cup C_2) - \{x,x'\} \subseteq W$, as desired.  The claim now follows by induction.
  \end{subproof}

  Now, \cref{induct} implies that the union of any $s$ pairs in $\pi$ contains a circuit, and the result follows.
\end{proof}

\begin{lemma}
  \label{lem:tis1}
  If $M$ is a matroid with the $(1,2,t,2t)$-property and at least $t$ elements, then $M$ is a $(1,t)$-spike.
  Dually, if $M$ is a matroid with the $(s,2s,1,2)$-property and at least $s$ elements, then $M$ is an $(s,1)$-spike.
\end{lemma}

\begin{proof}
By duality, it suffices to consider the case where $M$ has the $(1,2,t,2t)$-property and at least $t$ elements. Since every element of $M$ is contained in a $2$-element circuit, there is a partition of $E(M)$ into parallel classes $P_1,P_2,\ldots,P_n$, where $|P_i|\geq2$ for each $i$. For each $P_i$, let $x_i\in P_i$.

First, we consider the case where $n\geq t$. Let $X$ be a $t$-element subset of $\{x_1,\ldots,x_{n}\}$; for ease of notation, we assume $X=\{x_1,\ldots,x_{t}\}$. By the $(1,2,t,2t)$-property, $X\subseteq C^*$ for some $2t$-element cocircuit $C^*$. Since $P_i$ is a parallel class, $\{x_i,y_i\}$ is a circuit for each $y_i\in P_i-\{x_i\}$. By orthogonality, $y_i\in C^*$ for each such $y_i$, so $P_i\subseteq C^*$. Since $|C^*|=2t$, and $X$ is an arbitrary $t$-element subset of $\{x_1,\ldots,x_{n}\}$, it follows that $|P_i|=2$ for each $i\in[n]$, and that the union of any $t$ of the $P_i$'s is a cocircuit. Thus $M$ is a $(1,t)$-spike.

It remains to consider the case where $n<t$. Since $M$ has at least $t$ elements, let $X$ be any $t$-element set containing $\{x_1,\ldots,x_n\}$. By the $(1,2,t,2t)$-property, there is a $2t$-element cocircuit $C^*$ containing $X$. For $i\in[n]$ and each $y_i\in P_i-\{x_i\}$, orthogonality implies $y_i\in C^*$. Thus, $E(M)=C^*$. It follows that $M\cong U_{1,2t}$, which is a $(1,t)$-spike.
\end{proof}

We now prove \cref{mainthm}, restated below.
  
\begin{theorem}
   \label{mainthmtake2}
   There exists a function $f : \mathbb{N}^2 \rightarrow \mathbb{N}$ such that, if $M$ is a matroid with the \ttp and $|E(M)| \ge f(s,t)$, then $M$ is an $(s,t)$-spike.
 \end{theorem}

\begin{proof}
  If $s=1$ or $t=1$, then, by \cref{lem:tis1}, the theorem holds with $f(s,t) = \max\{s,t\}$.
  So we may assume that $\min\{s,t\} \ge 2$.
  A matroid is an $(s,t)$-spike if and only if its dual is a $(t,s)$-spike; moreover, a matroid has the \ttp if and only if its dual has the $(t,2t,s,2s)$-property. Therefore, by duality, we may also assume that $r(M)\geq r^*(M)$.
  
  Let $r_k(n)$ be the Ramsey number described in \cref{hyperramsey}.
  For $k \in [s]$, we define the function $h_k : \mathbb{N}^2 \rightarrow \mathbb{N}$ such that \[h_{s}(s,t)=\max\{s+2t-1,2s+t-1,3s+t-3,s+3t-3\}\]
  and such that $h_k(s,t)=r_k(h_{k+1}(s,t))$ for $k\in[s-1]$.
  Note that $h_{k}(s,t) \ge h_{k+1}(s,t) \ge h_{s}(s,t)$, for each $k \in [s-1]$.
    
    Let $p = h_1(s,t)$ and let $q(s,t)=\binom{2t}{2}^{s-1}(p + 2(s-1))$.
    By \cref{setup}, there exists a function $g$ such that if $|E(M)| \ge g(s,t,q(s,t))$, then $M$ has $s-1$ pairwise disjoint $2t$-element cocircuits $C_1^*, C_2^*, \dotsc, C_{s-1}^*$, and there is some $Z \subseteq E(M)-\bigcup_{i \in [s-1]}C_i^*$ such that
      $r_M(Z) \ge q(s,t)$, and,
      for each $z \in Z$, there exists an element $z'\in Z'-\{z\}$ such that $\{z,z'\}$ is contained in a $2s$-element circuit~$C$ with $|C \cap C_i^*|=2$ for each $i \in [s-1]$.
    
    Let $f(s,t) = g(s,t,q(s,t))$, and suppose that $|E(M)| \ge f(s,t)$.
    Then, by \cref{lem:payoff}, there exists a subset $Z \subseteq Z'$ such that $Z$ has a partition into pairs $\pi = ( Z_1, \dotsc, Z_{p})$ such that
    
  \begin{enumerate}[label=\rm(\Roman*)]
    \item each circuit of $M|Z$ is a union of pairs in $\pi$, and
    \item the union of any $s$ pairs in $\pi$ contains a circuit.\label{rc2}
  \end{enumerate}
  
  Let $m=h_{s}(s,t)$. By \cref{lem:swamping} and its dual,
  it suffices to show that $M$ has either an $s$-echidna or a $t$-coechidna of order $m$. 
  If the smallest circuit in $M|Z$ has size $2s$, then, by \ref{rc2}, $\pi$ is an $s$-echidna of order $p \ge m$.
  So we may assume that the smallest circuit in $M|Z$ has size $2j$ for some $j \in [s-1]$.
  
  \begin{sublemma}
    \label{iterramsey}
    If the smallest circuit in $M|Z$ has size $2j$, for $j \in [s-1]$, and $|\pi| \ge h_j(s,t)$, then either
    \begin{enumerate}
      \item $M$ has a $t$-coechidna of order $m$, 
      or\label{ir1}
      \item there exists some $Z' \subseteq Z$ that is the union of $h_{j+1}(s,t)$ pairs in $\pi$ for which the smallest circuit in $M|Z'$ has size at least $2(j+1)$.\label{ir2}
    \end{enumerate}
  \end{sublemma}
  
  \begin{subproof}
    We define $H$ to be the $j$-uniform hypergraph with vertex set $\pi$ whose hyperedges are the $j$-subsets of $\pi$ that are partitions of circuits in $M|Z$.
    By \cref{hyperramsey}, and the definition of $h_k$, as $H$ has at least $h_j(s,t)$ vertices, it has either a clique or a stable set, on $h_{j+1}(s,t)$ vertices.
    If $H$ has a stable set~$\pi'$ on $h_{j+1}(s,t)$ vertices, then clearly \ref{ir2} holds, with $Z' = \bigcup_{P \in \pi'} P$.

    Therefore, we may assume that there are $h_{j+1}(s,t)$ pairs in $\pi$ such that the union of any $j$ of these pairs is a circuit.
    Let $Z''$ be the union of these $h_{j+1}(s,t)$ pairs.
    We claim that the union of any set of $t$ pairs contained in $Z''$ is a cocircuit.
    Let $T$ be a transversal of $t$ pairs in $\pi$ contained in $Z''$, and let $C^*$ be the $2t$-element cocircuit containing $T$.
    Suppose, for a contradiction, that there exists some pair $P \in \pi$ with $P \subseteq Z''$ such that $|C^* \cap P| = 1$.
    Select $j-1$ pairs $Z_1'',\dotsc,Z_{j-1}''$ in $\pi$ that are each contained in $Z''-C^*$ (these exist since $h_{j+1}(s,t) \ge s+2t-1 \ge 2t + j - 1$).
    Then $P \cup (\bigcup_{i \in [j-1]}Z_i'')$ is a circuit intersecting $C^*$ in a single element, contradicting orthogonality.
    We deduce that the union of any $t$ pairs in $\pi$ that are contained in $Z''$ is a cocircuit.
    Thus, $M$ has a $t$-coechidna of order $h_{j+1}(t) \ge m$,
    satisfying \ref{ir1}.
  \end{subproof}

  We now apply \cref{iterramsey} iteratively, for a maximum of $s-j$ iterations.
  If \ref{ir1} holds, at any iteration, then $M$ has a $t$-coechidna of order $m$,
  as required.
  Otherwise, we let $\pi'$ be the partition of $Z'$ induced by $\pi$; then, at the next iteration, we relabel $Z=Z'$ and $\pi=\pi'$. 
  If \ref{ir2} holds for each of $s-j$ iterations, then we obtain a subset $Z'$ of $Z$ such that the smallest circuit in $M|Z'$ has size $2s$.
  Then, by \ref{rc2}, $M$ has an $s$-echidna of order $h_{s}(s,t)=m$,
  completing the proof.
\end{proof}

\section{Properties of \texorpdfstring{$(s,t)$}{(s,t)}-spikes}
\label{sec:tspikeprops}

In this section, we prove some properties of $(s,t)$-spikes.
In particular, we show that an $(s,t)$-spike has order at least $s+t-1$; an $(s,t)$-spike of order~$m$ has $2m$ elements and rank~$m+s-t$; and the circuits of an $(s,t)$-spike that are not a union of $s$ arms meet all but at most $t-2$ of the arms. We also give some results about the connectivity of $(s,t)$-spikes of sufficiently large order.

We also show that an appropriate concatenation of the associated partition of a $t$-spike is a $(2t-1)$-anemone, following the terminology of~\cite{ao2008}. Finally, we describe a construction that can be used to obtain an $(s,t+1)$-spike 
from an $(s,t)$-spike of sufficiently large order, and we show that every $(s,t+1)$-spike can be constructed from some $(s,t)$-spike in this way.

We again assume that $s$ and $t$ are positive integers.

\subsection*{Basic properties}

\begin{lemma}
  \label{tspikeorder}
  Let $M$ be an $(s,t)$-spike with associated partition $(A_1,\ldots,A_m)$.  Then $m \ge s+t-1$.
\end{lemma}
\begin{proof}
  By the definition of an $(s,t)$-spike, we have $m\geq\max\{s,t\}$. Let $Y = \bigcup_{j \in [t]}A_j$, and let $y\in Y$. Since $Y$ is a cocircuit, $Z=(E(M)-Y) \cup \{y\}$ spans $M$. Therefore, $r(M)\leq|Z|=2m-2t+1$. Similarly, by duality, $r^*(M)\leq2m-2s+1$. Therefore, \[2m = |E(M)| = r(M) + r^*(M) \le (2m-2t+1)+(2m-2s+1).\] The result follows.
\end{proof}
\begin{lemma}
    \label{lem:rank-matroid}
  Let $M$ be an $(s,t)$-spike of order~$m$.
  Then $r(M)=m+s-t$ and $r^*(M)=m-s+t$.
\end{lemma}
\begin{proof}
  Let $(A_1,\ldots,A_m)$ be the associated partition of $M$, and let $A_i = \{x_i,y_i\}$ for each $i \in [m]$. Choose $I \subseteq J \subseteq [m]$ such that $|I|=s-1$ and $|J| = m-t$. (This is possible by \cref{tspikeorder}.)
  Let $X = \{y_j : i \in I\} \cup \{x_j : j \in J\}$.
  Note that $\bigcup_{i \in I\cup J}A_i\subseteq\cl(X)$.
  Since $E(M)-\bigcup_{i \in I\cup J}A_i$ is a cocircuit, $\bigcup_{i \in I\cup J}A_i$ is a hyperplane.
  Therefore, $\bigcup_{i \in I\cup J}A_i=\cl(X)$, and we have $r(M)-1=r(X)\leq|X|=|I|+|J|=m+s-t-1$. Thus, $r(M)\leq m+s-t$. Similarly, by duality, $r^*(M)\leq m-s+t$.
  
  Therefore, we have \[2m=|E(M)|=r(M)+r^*(M)\leq(m+s-t)+(m-s+t)=2m.\] Thus, we must have equality, and the result holds.
\end{proof}

\sloppy
\begin{lemma}
  \label{l:circuits}
  Let $M$ be an $(s,t)$-spike of order~$m$ with associated partition $(A_1,\ldots,A_m)$, and
  let $C$ be a circuit of $M$.
  \begin{enumerate}
    \item
      $C = \bigcup_{j \in J}A_j$ for some $s$-element set $J \subseteq [m]$, or\label{c1}
    \item
      $\left|\{i \in [m] : A_i \cap C \neq \emptyset\}\right| \ge m-(t-2)$ and
      $\left|\{i \in [m] : A_i \subseteq C\}\right| < s$.\label{c2}
  \end{enumerate}
\end{lemma}
\fussy
\begin{proof}
  Let $S = \{i \in [m] : A_i \cap C \neq \emptyset\}$. Thus, $S$ is the minimal subset of $[m]$ such that $C \subseteq \bigcup_{i \in S}A_i$.
  We have $|S| \ge s$ since $C$ is independent otherwise.
  If $|S|=s$, then $C$ satisfies \ref{c1}.
  Therefore, we may assume $|S| > s$.
  We must have $\left|\{i \in [m] : A_i \subseteq C\}\right| < s$; otherwise $C$ properly contains a circuit.
  Thus, there is some $j \in S$ such that $A_j - C \neq \emptyset$.
  If $|S| \ge m-(t-2)$, then $C$ satisfies \ref{c2}.
  Therefore, we may assume $|S| \le m-(t-1)$.
  Let $T = ([m]-S) \cup \{j\}$.  Then $|T|\ge t$, implying that $\bigcup_{i \in T}A_i$ contains a cocircuit intersecting $C$ in one element.
  This contradicts orthogonality.
\end{proof}

In the remainder of the paper, if $(A_1,\ldots,A_m)$ is the associated partition of an $(s,t)$-spike and $J\subseteq[m]$, then we define \[A_J=\bigcup_{j \in J} A_j.\]

\begin{proposition}
  \label{pro:rank-func}
  Let $\pi=(A_1,\ldots,A_m)$ be the associated partition of an $(s,t)$-spike. If $J\subseteq[m]$, then
  \[r(A_J) = 
    \begin{cases}
      2|J| & \textrm{if $|J| < s$,}\\
      s+|J|-1 & \textrm{if $s\leq|J| \leq m-t+1$,}\\
      m+s-t & \textrm{if $|J| \ge m-t+1$.}
    \end{cases}\]
\end{proposition}

\begin{proof}
  If $|J|<s$, then $A_J$  is properly contained in a circuit and is therefore independent. Thus, $r(A_J)=|A_J|=2|J|$.
  
  We now prove that $r(A_J)=s+|J|-1$ if $s\leq|J| \leq m-t+1$. We proceed by induction on $|J|$. As a base case, if $|J|=s$, then $A_J$ is a circuit. Therefore, $r(A_J)=|A_J|-1=s+|J|-1$. Now, for the inductive step, let $s<|J|\leq m-t+1$, and let $J'\subseteq J$ with $|J'|=|J|-1$. By induction, $r(A_{J'})=s+|J|-2$. Let $\{x_i,y_i\}=A_J-A_{J'}$. By \cref{l:circuits}, since $|J|<m-t+2$, there is no circuit $C$ such that $x_i\in C\subseteq A_{J'}\cup\{x_i\}$. Therefore, $x_i\notin\cl(A_{J'})$, and $r(A_{J'}\cup\{x_i\})=r(A_{J'})+1$. On the other hand, since $|J|>s$, there is a circuit $C$ such that $y_i\in C\subseteq A_{J}$. Therefore, $y_i\in\cl(A_{J'}\cup\{x_i\})$, and $r(A_J)=r(A_{J'})+1=s+|J|-1$.
  
  Note that the preceding argument, along with \cref{lem:rank-matroid} implies that, if $|J|=m-t+1$, then $A_J$ is spanning. Thus, if $|J|\geq m-t+1$, then $r(A_J)=r(M)=m+s-t$.
\end{proof}

\subsection*{Connectivity}

Let $M$ be a matroid with ground set $E$.
Recall that the \emph{connectivity function} of $M$, denoted by $\lambda$, is defined as 
\begin{align*}
  \lambda(X) = r(X) + r(E - X) - r(M),
\end{align*}
for all subsets $X$ of $E$. In the case where $M$ is an $(s,t)$-spike of order $m$ and $X=A_J$ for some set $J\subseteq[m]$, this implies
\begin{align*}
  \lambda(A_J) = r(A_J) + r(A_{[m]-J}) - r(M).
\end{align*}

Therefore, \cref{pro:rank-func} allows us to easily compute $\lambda(A_J)$.

\begin{lemma}
  \label{lem:conn}
  Let $\pi=(A_1,\ldots,A_m)$ be the associated partition of an $(s,t)$-spike, and let $(J,K)$ be a partition of $[m]$ with $|J| \le |K|$.
  \begin{enumerate}
      \item If $|J|\leq t-1$, then $\lambda(A_J)=r(A_J)$.
      \item If $t-1\leq|J|\leq m-s$, then \[\lambda(A_J)=
      \begin{cases}
      t+|J|-1 & \textrm{if $|J| < s$,}\\
      s+t-2 & \textrm{if $s\leq|J|\leq m-t+1$.}
    \end{cases}\]
    \item If $|J|> m-s$, then $\lambda(A_J)=m-s+t$.
  \end{enumerate}
\end{lemma}

\begin{proof}
If $|J|\leq t-1$, then $|K|\geq m-t+1$. Therefore, $A_K$ is spanning, and $\lambda(A_J)=r(A_J)+r(A_K)-r(M)=r(A_J)$. Statement (i) follows.

If $t-1\leq|J|\leq m-s$, then $s\leq|K|\leq m-t+1$. Therefore, $\lambda(A_J)=r(A_J)+r(A_K)-r(M)=r(A_J)+s+m-|J|-1-(m+s-t)$. Statement (ii) follows. (Note that we cannot have $|J|>m-t+1$ because otherwise $|K|<t-1\leq|J|$.)

If $|J|> m-s$, then $s>|K|\geq|J|$. Therefore, $\lambda(A_J)=r(A_J)+r(A_K)-r(M)=2|J|+2(m-|J|)-(m+s-t)=m-s+t$. Statement (iii) follows. 
\end{proof}

Using the terminology of~\cite{ao2008}, \cref{lem:conn} implies the following.

\begin{proposition}
\label{pro:anemone}
Let $(A_1,\dotsc,A_m)$ be the associated partition of an $(s,t)$-spike~$M$, and suppose that $(P_1,\dotsc,P_k)$ is a partition of $E(M)$ such that, for each $i \in [k]$, $P_i = \bigcup_{i \in I}A_i$ for some subset $I$ of $[m]$, with $|I| \ge \max\{s-1,t-1\}$. Then $(P_1,\dotsc,P_k)$ is an $(s+t-1)$-anemone.
\end{proposition}

We now continue our study of the connectivity of $(s,t)$-spikes.

\begin{lemma}
\label{ind-and-coind}
Let $M$ be an $(s,t)$-spike of order $m\geq3\max\{s,t\}-2$, and let $X\subseteq E(M)$ such that $|X|\leq2\min\{s,t\}-1$. Then $\lambda(X)=|X|$.
\end{lemma}

\begin{proof}
By Lemma \ref{l:circuits}, if $X$ is dependent, then either $|X|=2s$ or $|X|\geq m-t+2\geq 3\max\{s,t\}-2-t+2=3\max\{s,t\}-t\geq2\max\{s,t\}\geq2s$. However, $|X|\leq2\min\{s,t\}-1<2s$. Therefore, $X$ is independent, which implies that $r(X)=|X|$.

By a similar argument, using the dual of \cref{l:circuits}, $X$ is coindependent, implying that $r(E(M)-X)=r(M)$. Therefore, 
  \begin{align*}
  \lambda(X)&=r(X)+r(E(M)-X)-r(M)\\
  &=|X|+r(M)-r(M)\\
  &=|X|,
  \end{align*} proving the lemma.
\end{proof}

\begin{theorem}
Let $M$ be an $(s,t)$-spike of order \[m\geq\max\{3s+t,s+3t\}-4,\] where $\min\{s,t\}\geq2$. Then $M$ is $(2\min\{s,t\}-1)$-connected.
\end{theorem}

\begin{proof}
Because $M^*$ is a $(t,s)$-spike and because $\lambda_{M^*}=\lambda_M$, we may assume without loss of generality that $t\leq s$. Note that $\max\{3s+t,s+3t\}=3\max\{s,t\}+\min\{s,t\}$. Therefore, $m\geq3s+t-4$, and we must show that $M$ is $(2t-1)$-connected.

Now, suppose for a contradiction that $M$ is not $(2t-1)$-connected. Then there is a $k$-separation $(P,Q)$ of $M$, with $|P|\geq|Q|$, for some $k<2t-1$. Therefore, $\lambda(P)=\lambda(Q)<k\leq2t-2$.

First, we consider the case where $A_I \subseteq P$, for some $(t-1)$-element set $I \subseteq [m]$. Let $U = \{u \in [m] : |P \cap A_u|= 1\}$. Then $A_j \subseteq \cl_{M^*}(P)$ for each $j \in U$. For such a $j$, it follows, by the definition of $\lambda_{M^*}$ (which is equal to $\lambda_M=\lambda$), that $\lambda(P \cup A_j) \le \lambda(P)$. We use this repeatedly below; in particular, we see that $\lambda(P\cup A_U)\leq\lambda(P)$.

Let $P' = P\cup A_U$, and let $Q' = E(M)-P'$. Then there is a partition $(J,K)$ of $[m]$, with $|J|\leq|K|$, such that $Q'=A_J$ and $P'=A_K$. Moreover, $\lambda(Q')=\lambda(P')\leq\lambda(P)$.

Suppose $|J|\geq t-1$. Note that $m\geq3s+t-4\geq2s$ since $\min\{s,t\}\geq2$. Therefore, $|J|\leq\frac{1}{2}m=m-\frac{1}{2}m\leq m-\frac{1}{2}(2s)=m-s$. Thus, to determine $\lambda(Q')$, we need only consider Lemma \ref{lem:conn}(ii). If $|J|\geq s$, then by Lemma \ref{lem:conn}(ii), \[\lambda(P)\geq\lambda(P')=\lambda(Q')=s+t-2\geq2t-2,\] a contradiction. Otherwise, $|J|<s$, implying by Lemma \ref{lem:conn}(ii) that \[\lambda(P)\geq\lambda(P')=\lambda(Q')=t+|J|-1\geq t+t-1-1=2t-2,\] another contradiction.

Therefore, $|J|<t-1$. Let
$U'\subseteq U$ such that $|U'|=|Q|-(2t-2)$. Then $\lambda(P) \ge \lambda\left(P \cup A_{U'}\right) = \lambda\left(Q- A_{U'}\right)$. Since $\left|Q- A_{U'}\right| = 2t-2$ and $m\geq3s+t-4\geq3s-2$, \cref{ind-and-coind} implies that $\lambda\left(Q-A_{U'}\right)=2t-2$, so $\lambda(P) \ge 2t-2$, a contradiction.

Now we consider the case that $|\{i \in [m] : A_i \subseteq P\}| < t-1$. Since $|Q| \le |P|$, it follows that $|\{i \in [m] : A_i \subseteq Q\}| \le |\{i \in [m] : A_i \subseteq P\}| < t-1<s$.

Now, since $|\{i \in [m] : A_i \subseteq P\}| < t-1$, we have $|\{i \in [m] : A_i \cap Q \neq \emptyset\}| > m-(t-1)$. Therefore, $r(Q) \ge m-(t-1)$ by \cref{l:circuits}. Similarly, $r(P) \ge m-(t-1)$. Thus, 
  \begin{align*}
    \lambda(P) &= r(P) + r(Q) - r(M) \\
    &\ge (m-(t-1)) + (m-(t-1)) - (m+s-t) \\
    &=m-s-t+2 \\
    &\ge 3s+t-4-s-t+2 \\ 
    &= 2s-2\\
    &\ge 2t-2,
  \end{align*}
  a contradiction. This completes the proof.
\end{proof}

\subsection*{Constructions}
In \cite{bccgw2019}, a construction is described that, starting from a $(t,t)$-spike $M_0$, obtains a $(t+1,t+1)$-spike $M_1$. This construction consists of a certain elementary quotient $M_0'$ of $M_0$, followed by a certain elementary lift $M_1$ of $M_0'$. It is shown in \cite{bccgw2019} that $M_1$ is a $(t+1,t+1)$-spike as long as the order of $M_0$ is sufficiently large. 

In the process of constructing $M_1$ in this way, the intermediary matroid $M_0'$ is a $(t,t+1)$-spike. For the sake of completeness, we will review this construction in the more general case where $M_0$ is an $(s,t)$-spike, in which case $M_0'$ is an $(s,t+1)$-spike. To construct an $(s+1,t)$-spike, we perform the construction on $M^*$ and dualize. Since $(2,2)$-spikes (and indeed, $(1,1)$-spikes) are well known to exist, this means that $(s,t)$-spikes exist for all positive integers $s$ and $t$.

It is also shown in \cite{bccgw2019} that all $(t,t)$-spikes can be constructed in this manner. We also extend this to the general case of $(s,t)$-spikes below.

Recall that $M_1$ is an \emph{elementary quotient} of $M_0$ if there is a single-element extension $M^+_0$ of $M_0$ by an element~$e$ such that $M_1 = M^+_0 / e$.
If $M_1$ is an elementary quotient of $M_0$, then $M_0$ is an \emph{elementary lift} of $M_1$. Also, note that if $M_1$ is an elementary lift of $M_0$, then $M_1^*$ is an elementary quotient of $M_0^*$.

\begin{construction}
\label{cons:quotient}
Let $M$ be an $(s,t)$-spike of order~$m \ge s+t$, with associated partition $\pi$.
Let $M+e$ be a single-element extension of $M$ by an element $e$ such that $e$
blocks each $2t$-element cocircuit that is a union of $t$ arms of $M$.  Then let $M'=(M+e)/e$.
\end{construction}

In other words, $M+e$ has the property that $e\notin \cl_{M+e}(E(M)-C^*)$ for every $2t$-element cocircuit $C^*$ that is the union of $t$ arms.
Note that one possibility is that $M+e$ is the free extension of $M$ by an element $e$.
Since $m-t\geq s$, we have $e\notin\cl_{M+e}(C)$ for each $2s$-element circuit $C$. Thus, in $M'$, the union of any $s$ arms of the $(s,t)$-spike $M$ is still a circuit of $M'$. However, since $r(M') = r(M) - 1$, the union of any $t+1$ arms is a $2(t+1)$-element cocircuit. Therefore, $M'$ is an $(s,t+1)$-spike.

Note that $M'$ is not unique; more than one $(s,t+1)$-spike can be constructed from a given $(s,t)$-spike $M$ using \cref{cons:quotient}. Given an $(s+1,t)$-spike~$M'$, we will describe how to obtain an $(s,t)$-spike~$M$ from $M'$ by a specific elementary quotient. This process reverses the dual of \cref{cons:quotient}. This will then imply that every $(s,t)$-spike can be constructed from a $(1,1)$-spike by repeated use of \cref{cons:quotient} and its dual. \cref{modcut} describes the single-element extension that gives rise to the elementary quotient we desire.
Intuitively, the extension adds a ``tip'' to the $(s,t)$-spike. In the proof of this lemma, we assume knowledge of the theory of modular cuts (see \cite[Section~7.2]{oxbook}).

The proof of \cref{modcut} will be very similar to the proof of \cite[Lemma 6.6]{bccgw2019}. However, we note that \cite[Lemma 6.6]{bccgw2019} is falsely stated; what is proven in \cite{bccgw2019} is essentially the specialisation of \cref{modcut}, below, in the case that $s=t$.
The statement of \cite[Lemma 6.6]{bccgw2019} replaces the condition that $M$ is a $(t,t)$-spike with the weaker condition that $M$ has a $t$-echidna. To demonstrate that this is overly general, consider the rank-$3$ matroid consisting of two disjoint lines with four points. Let these lines be $\{a,b,c,d\}$ and $\{w,x,y,z\}$. Then $(\{a,b\},\{w,x\})$ is a $2$-echidna of order $2$. For \cite[Lemma 6.6]{bccgw2019} to be true, we would need a single-element extension $M^+$ by an element $e$ such that $e\in\cl_{M^+}(\{a,b\})$ but $e\notin\cl_{M^+}(\{c,d\})$. This is impossible since $\cl_M(\{a,b\})=\cl_M(\{c,d\})$.

\begin{lemma}
  \label{modcut}
  Let $M$ be an $(s,t)$-spike.
  There is a single-element extension $M^+$ of $M$ by an element $e$ having the property that, for every $X \subseteq E(M)$, $e \in \cl_{M^+}(X)$ if and only if $X$ contains at least $s-1$  arms of $M$.
\end{lemma}

\begin{proof}
Since $M$ is an $(s,t)$-spike, there is a partition $\pi=(S_1,\dotsc,S_m)$ of $E(M)$ that is both an $s$-echidna and a $t$-coechidna. Let $$\mathcal{F} = \left\{\bigcup_{i\in I}S_i : I \subseteq [m] \textrm{ and } |I|=s-1\right\}.$$ By the definition of an $s$-echidna, $\mathcal{F}$ is a collection of flats of $M$. Let $\mathcal{M}$ be the set of all flats of $M$ containing some flat $F \in \mathcal{F}$. We claim that $\mathcal{M}$ is a modular cut. Recall that, for distinct $F_1,F_2 \in \mathcal{M}$, the pair $(F_1,F_2)$ is \emph{modular} if $r(F_1) + r(F_2) = r(F_1 \cup F_2) + r(F_1 \cap F_2)$. To show that $\mathcal{M}$ is a modular cut, it suffices to prove that, for any $F_1,F_2 \in \mathcal{M}$ such that $(F_1,F_2)$ is a modular pair, $F_1 \cap F_2 \in \mathcal{M}$.

For any $F \in \mathcal{M}$, since $F$ contains at least $s-1$ arms of $M$, and the union of any $s$ arms is a circuit, it follows that $F$ is a union of arms of $M$. Thus, let $F_1,F_2 \in \mathcal{M}$ be such that $F_1=\bigcup_{i\in I_1}S_i$ and $F_2=\bigcup_{i\in I_2}S_i$, where $I_1$ and $I_2$ are distinct subsets of $[m]$ with $u_1=|I_1| \ge s-1$ and $u_2=|I_2|\ge s-1$. 

Let $q=|I_1 \cap I_2|$. Then $F_1 \cup F_2$ is the union of $u_1 + u_2 - q \ge s-1$ arms, and $F_1\cap F_2$ is the union of $q$ arms. We show that if $q<s-1$, then $(F_1,F_2)$ is not a modular pair.

We consider several cases. First, suppose $u_1,u_2\leq m-t+1$. By \cref{pro:rank-func}, \begin{align*}
    r(F_1) + r(F_2) &= (s + u_1 - 1) + (s + u_2 - 1) \\
    &>
    (s-1 + u_1 + u_2 - q) +2q \\
    &= s+|I_1\cup I_2|-1+2|I_1\cap I_2| \\
    &\geq r(F_1 \cup F_2) + r(F_1 \cap F_2).
  \end{align*}
  
 Next, consider the case where $u_2\leq m-t+1<u_1$. (By symmetry, the argument is the same if $u_1$ and $u_2$ are swapped.) One can check that  $u_1+u_2-q>m-t+1$. By \cref{pro:rank-func}, \begin{align*}
    r(F_1) + r(F_2) &= (m+s-t) + (s + u_2 - 1) \\
    &> (m + s-t)+2q\\
    &= r(F_1 \cup F_2) + r(F_1 \cap F_2).
    \end{align*}
    
Finally, consider the case where $u_1,u_2>m-t-1$. We have 
    \[r(F_1) + r(F_2) = 2m+2s -2t,\] which by \cref{tspikeorder}, is at least
    \begin{align*}
    m+3s-t-1
    &> m+s-t+2q\\
    &= r(F_1 \cup F_2) + r(F_1 \cap F_2).
    \end{align*}
    
Thus, in all cases, $(F_1,F_2)$ is not a modular pair. Therefore, we have shown that $\mathcal{M}$ is a modular cut. Now, there is a single-element extension corresponding to the modular cut~$\mathcal{M}$, and this extension satisfies the requirements of the lemma (see, for example, \cite[Theorem~7.2.3]{oxbook}).
\end{proof}

\begin{theorem}
Let $M$ be an $(s,t)$-spike of order $m\geq s+t$. Then $M$ can be constructed from a $(1,1)$-spike of order $m$ by applying \cref{cons:quotient} $t-1$ times, followed by the dual of \cref{cons:quotient} $s-1$ times.
\end{theorem}

\begin{proof}
For $s=t=1$, the result is clear. Otherwise, by duality, we may assume without loss of generality that $t>1$. By induction and duality, it suffices to show that $M$ can be constructed from an $(s-1,t)$-spike of order $m$ by applying the dual of \cref{cons:quotient} once.

Let $\pi=(A_1,\dotsc,A_m)$ be the associated partition of $M$. Let $M^+$ be the single-element extension of $M$ by an element~$e$ described in \cref{modcut}.

Let $M'=M^+/e$. We claim that $\pi$ is an $(s-1)$-echidna and a $t$-coechidna that partitions the ground set of $M'$.

Let $X$ be the union of any $s-1$ spines of $\pi$.  Then $X$ is independent in $M$, and $X \cup \{e\}$ is a circuit in $M^+$, so $X$ is a circuit in $M'$.
Thus, $\pi$ is an $(s-1)$-echidna of $M'$. Now let $C^*$ be the union of any $t$ spines of $\pi$, and let $H=E(M)-C^*$.  Then $H$ is the union of at least $s-1$ spines, so $e \in \cl_{M^+}(H)$.  Now $H \cup \{e\}$ is a hyperplane in $M^+$, so $C^*$ is a cocircuit in $M^+$ and therefore in $M'$. Hence $\pi$ is a $t$-coechidna of $M'$.

Note that $M'$ is an elementary quotient of $M$, so $M$ is an elementary lift of $M'$ where none of the $2(s-1)$-element circuits of $M'$ are preserved in $M$. So the $(s,t)$-spike $M$ can be obtained from the $(s-1,t)$-spike $M'$ using the dual of \cref{cons:quotient}.
\end{proof}

\section*{Acknowledgements}
Work for this project was begun during a visit, funded by the Heilbronn Institute for Mathematical Research, at the University of Bristol, United Kingdom.

\sloppy

\end{document}